\documentclass{amsart}
\usepackage[utf8]{inputenc}
\usepackage[pdftex]{graphicx}
\usepackage{epstopdf}
\usepackage{graphics,psfrag,subfigure,epsfig}
\usepackage{caption}
\usepackage{amssymb}
\usepackage{latexsym}
\usepackage{xypic}
\usepackage{multirow}

\newtheorem{theorem}{Theorem}[section]
\newtheorem{lemma}[theorem]{Lemma}
\newtheorem{proposition}[theorem]{Proposition}

\newtheorem{corollary}[theorem]{Corollary}

\theoremstyle{definition}

\newtheorem{example}[theorem]{Example}

\theoremstyle{remark}

\numberwithin{equation}{section}

\def\C{{\mathbb C}}

\def\P{\mathbb P_{\mathbb C}}

\def\PSL{\text{{PSL}}}
\def\SL{\text{{SL}}}

\def\SP{\text{{SP}}}

\def\GL{\text{{GL}}\,}
\def\GL{\text{{GL}}\,}

\title[Exceptional sets]{Exceptional Algebraic sets for infinite discrete groups of $PSL(3,\Bbb{C})$} 

\author{Angel Cano}
\address{ UCIM, Av. Universidad s/n. Col. Lomas de Chamilpa, C.P. 62210, Cuernavaca, Morelos, M\'exico.}
\email{angelcano@im.unam.mx}

\author{ Luis Loeza }
\address{ IIT UACJ, Av. del Charro no. 450 Nte. Col. Partido Romero CP 32310,  Ciudad Juárez, Chihuahua, M\'exico.}

\email{luis.loeza@uacj.mx}

\thanks{Partially supported by grants of projects PAPPIT UNAM:  IN110219} 

\subjclass{Primary 37F99,  Secondary 30F40, 20H10, 57M60}

\begin{document}

	\begin{abstract}
In this  note we show that the exceptional algebraic  set for   an infinite discrete group in $PSL(3,\C)$ should be a finite union of: complex lines, copies of the Veronese curve or copies of the cubic $xy^2-z^3$.
	\end{abstract}

	\maketitle 
	

	\section*{Introduction}
	
	Complex Kleinian groups first appeared in mathematics with the works of Henri Poincaré, as a way to qualitatively study the solutions of ordinary differential equations of order two, one can say that the success of  Poincaré was because he  managed to establish a dictionary between differential equations and group actions. Subsequently, this theory achieved a new boom with the introduction of quasi-conformal maps and the discovering of bridges between hyperbolic three-manifolds this theory. At the beginning of the '90s  A. Verjovsky and J. Seade began studying, see \cite{SV}, the discrete groups of projective transformations that act in projective spaces, as a proposal to establish a dictionary between actions of discrete groups and the theory of foliations, partial and  ordinary differential equations. The purpose of this note is to understand those infinite discrete groups  whose dynamic could be described in terms of a group  acting in an algebraic curve  in the complex-projective plane,{\it i. e.} Kleinian groups leaving an  invariant algebraic curve,  analogous results but in the case of iteration holomorphic maps in $\P^n$ have been studied extensively, see \cite{BCS,CLi,FS}.

	Recall that in the classical case,  see  \cite{Gr}, a set in $\P^1$ is said to be exceptional for the action of an infinite  group $G\subset \PSL(2,\C)$, discrete or not, if it is invariant under the action of $G$ and is a finite set. In analogy with this, we will say that $S\subset P^2$ is an  {\it exceptional algebraic set}  for the action of an infinite group, in our case discrete, if it is $G$-invariant and a complex algebraic curve, compare with the definition of algebraically mixing in \cite{SV}. As we will see, the geometry of  the exceptional algebraic sets  is very restricted and in consequence, the class of groups with an exceptional set is   small, more precisely in this article, we show:

		\begin{theorem} \label{t:main1}
		Let $G\subset PSL(\C)$ be an infinite discrete group and $S$ a complex algebraic surface invariant under $G$, if $S_0$ is an irreducible component of $S$, then  $S_0$ is a complex line, the Veronese curve or the projective curve induced by the polynomial $p(x,y,z)=xy^2-z^3$.
		\end{theorem}
	\begin{theorem} \label{t:main2}
		Let $G\subset PSL(\C)$ be an  infinite discrete group. If  $G$ is a non-virtually commutative group and  $S$ is a complex algebraic surface invariant under $G$, then $S$ is either a finite union of lines  or 
	 the Veronese curve. Moreover, if $S$ has at least four lines, then  $S$ contains at most three non-concurrent lines.
	 	\end{theorem}
		\begin{theorem} \label{t:main3}
		Let $G\subset PSL(3,\C)$ be infinite  discrete  groups. If  $G$ is virtually cyclic group and   $S$ a complex algebraic surface invariant under $G$, then $S$ is:
		\begin{enumerate}
			\item  A finite union of lines. Moreover, in  such a  collection of lines    the largest number of non-concurrent lines is three. 
						\item  A finite union of copies of the  Veronese curve and a finite union (possibly empty) of lines 
						which are  either tangent and secant to the copies of the Veronese curve.  Moreover the number of lines  which are tangents does not exceed two and the number of secants is at most one. 
				\item  A finite union of copies of the cubic induced by the polynomial $xy^2-z^3$ and  finite union (possibly empty) on tangent and secant lines.  Moreover the number of tangents does not exceed two and the number of secants is at most one. 
		\end{enumerate}
	\end{theorem}
	\begin{corollary} \label{c:main}
	Let $G\subset PSL(3,\C)$ be a discrete group with an algebraic exceptional set, then $G$ is either  virtually affine or is the  representation of a Kleinian group of  Möbius transformations  though  the irreducible representation of $Mob(\hat \C)$ into $PSL(3,\C)$, see example   \ref{e:ver}  below.
	\end{corollary}
	
	The paper is organized as follows: Section \ref{s:nb} reviews some elementary well-known facts on two-dimensional complex Kleinain groups, see \cite{CNS} for  extra information. In Section   \ref{s:examples}, we provide a collection of   examples that  depicts all the possible ways to constructing groups with exceptional algebraic surfaces as well as the respective surfaces.  
		In Section \ref{GDC} we show that every irreducible component of an exceptional algebraic set is a line, a copy of the Veronese curve of a cubic with a cusp as a singularity, the proof of this fact relies strongly in the use of the Plücker Formulas as well as in the Hurwitz theorems for curves, most of the material on  algebraic curves used in this note  is well known but the interested reader could see \cite{Fischer, miranda, sha} for full details. Finally  in section \ref{s:main}  we prove the main theorem of this article.

	\section{Preliminaries } \label{s:nb}
	In  this section we  establish some elementary facts  that we use in the sequel.
	
	\subsection{Projective Geometry}
	The complex projective plane
	$\mathbb{P}^2_{\mathbb{C}}$ is the quotient space 
	$(\mathbb{C}^{3}-
	\{{\bf 0}\})/\mathbb{C}^*
	,$
	where $\mathbb{C}^*$ acts on $\mathbb{C}^3-\{{\bf 0}\}$ by the
	usual scalar
	multiplication. Let $[\mbox{
	}]:\mathbb{C}^{3}-\{{\bf 0}\}\rightarrow
	\mathbb{P}^{2}_{\mathbb{C}}$ be the quotient map. A set
	$\ell\subset \mathbb{P}^2_{\mathbb{C}}$ is said to be a
	complex line if $[\ell]^{-1}\cup \{{\bf 0}\}$ is a complex linear
	subspace of dimension $2$. Given  $p,q\in
	\mathbb{P}^2_{\mathbb{C}}$ distinct points, there exists  a unique
	complex line passing through  $p$ and $q$, such line is 
	denoted by $\overleftrightarrow{p,q}$. The set of all  
	complex lines   in $\P^2$, denoted $Gr(\P^2)$, equipped
	with the topology of the Hausdorff convergence, actually is 
	diffeomorphic to  $\P^2$ and it is its projective dual $ \P^{*2} \cong Gr(\P^2)$.

	Consider the action of $\mathbb{Z}_{3}$ (viewed as the cubic
	roots of
	the unity) on  $\SL(3,\mathbb{C})$ given by the usual scalar
	multiplication. Then
	$$\PSL(3,\mathbb{C})=\SL(3,\mathbb{C})/\mathbb{Z}_{3}\,,$$ is a
	Lie group
	whose elements are called projective transformations.  Let
	$[[\mbox{ }]]:\SL(3,\mathbb{C})\rightarrow \PSL(3,\mathbb{C})$ be
	the
	quotient map, $g\in \PSL(3,\mathbb{C})$ and
	${\bf g}\in
	\GL(3,\mathbb{C})$. We say that ${\bf g}$ is a lift of
	$g$ if there exists a cubic root $\alpha$ of $Det({\bf g})$ such
	that $[\alpha^{-1} {\bf g}]=g$, by abuse of notation in the following we will use $[\mbox{ }] $ instead $[[\mbox{ }]] $ . We use the notation
	$(g_{ij})$ to denote elements in $\SL(3,\Bbb{C})$. It is easy to  show that
	$ \PSL(3,\mathbb{C})$  acts  transitively,
	effectively and by biholomorphisms on
	$\mathbb{P}^2_{\mathbb{C}}$
	by $[{\bf g }]([w])=[{\bf g }(w)]$, where $w\in
	\mathbb{C}^3-\{{\bf 0}\}$ and ${\bf g }\in \GL(3,\mathbb{C})$.\\
	
	If  $g$ is an element  in $\PSL(3,\C)$ and ${\bf g}\in \SL(3,\C)$ is a lift of $g$,  we say , see \cite{CNS},  that:
	
	\begin{itemize}
		\item $g$ is a elliptic  if ${\bf g }$ is diagonalizable with unitary eigenvalues. 
		
		\item $g$ is parabolic   if ${\bf g }$ is non-diagonalizable with unitary eigenvalues. 
		
		\item $g$ is loxodromic    if  ${\bf g }$ has some non-unitary eigenvalue.
	\end{itemize}
Clearly this definition does not depend on the choice of the lift.\\

\begin{lemma} [See Lemma 6.6 in \cite{CS}]  
	 \label{l:inford}
	Let $G$ be an infinite discrete group, then $G$ contains an element $g$  with  infinite order.  Moreover, $g$ is either  parabolic or loxodromic.
\end{lemma}
It is possible  to decide if the element in the previous lemma  is parabolic or loxodromic but one should impose dynamic restrictions over  $G$, a detailed discussion of this conditions   goes beyond the scope of this paper, the interested reader see \cite{BCNS1}.

Let $M(3,\C) $ be the set of all $3 \times
3$ matrices with complex coefficients.
Define the space of  pseudo-projective maps by:
$\SP(3,\C)=(M(3,\C)-\{{\bf 0}\})/\mathbb{C}^*,$
where $\mathbb{C}^*$ acts on $M(3,\C)-\{{\bf 0}\}$ by the
usual scalar
multiplication. We have the  quotient map $[\mbox{ }]:M(3,\C)-\{{\bf 0}\}\rightarrow
\SP(3,\C)$. Given
$P\in \SP(3,\C)$ we define its kernel by:
$$Ker(P)=[Ker({\bf P})-\{{\bf 0}\}],$$
where ${\bf P }\in M(3,\C)$ is a lift of $P$.   
Clearly $\PSL(3,\Bbb{C})\subset \SP(3,\Bbb{C})$ and an element $P$ in $\SP(3,\C)$ is in 
$\PSL(3,\C)$ if and only if $Ker(P)=\emptyset$. Notice that $\SP(3,\C)$ is a  manifold naturally diffeomorphic to $\P^8$, so it is compact.  For each $P\in \SP(3,\C)$  with lift $p\in M(3,\C)-\{{\bf 0}\}$ we can define  an holomorphic funtion $P:\P^2-Ker(P)\rightarrow \P^2$ by $P[w]=[p(w)]$.
The following lemma relates the convergence of projective maps as pseudo-projective maps and the converge holomorphic functions, also can  be considered as a generalization of the convergence property of Möbius transformations, see page 44 in  \cite{Kap}.

\begin{lemma}[See Proposition 1.1 in  \cite{CLU} ]	
	Let $(g_n)\subset PSL(3,\C)$ be a sequence of projective transformations, then  there is a subsequence $(h_n)$ of $(g_n)$ and $h\in \SP(3,\C)$ such that:  
	\begin{enumerate}
		\item The sequence  $h_n$ converges to $h$ as elements in  $\SP(3,\C)$.
		\item  Considering $h_n$  and $h$ as holomorphic functions from $\P^2-Ker(h)$ to $\P^2$ we have $h_n$ converges to $h$  uniformly on compact sets of $\P^2-Ker(h)$.
		\end{enumerate}
\end{lemma}

The following proposition  is   essentially due in \cite{Gus}. 

\begin{proposition}	 [See Lemmas  5.5  and 6.10 in \cite{Gus}]  \label{p:eq}  

	 \begin{enumerate}
	 	\item[] 
\item 	$
	\begin{bmatrix}
	1 & 1 &0\\
	0 &1 & 1\\
	0 & 0 & 1 
	\end{bmatrix}^n$ converges  to 	$
	\begin{bmatrix}
	0 & 0 &1\\
	0 &0  & 0\\
	0 & 0 & 0 
	\end{bmatrix}$. 

	\item  $
\begin{bmatrix}
\alpha & 1 &0\\
0 &\alpha  & 0\\
0 & 0 & \alpha^{-2}  
\end{bmatrix}^n$ converges  to 	$
\begin{bmatrix}
0 & 1 &0\\
0 &0  & 0\\
0 & 0 & 0 
\end{bmatrix}, \textrm{ if  } \vert \alpha\vert\geq 1 $.

\item  $
\begin{bmatrix}
\alpha & 1 &0\\
0 &\alpha  & 0\\
0 & 0 & \alpha^{-2}  
\end{bmatrix}^n$ converges  to 	$
\begin{bmatrix}
0 & 0 &0\\
0 &0  & 0\\
0 & 0 & 1 
\end{bmatrix}, \textrm{ if  } \vert \alpha\vert<1 $.

\item  $
\begin{bmatrix}
\alpha & 0 &0\\
0 &\beta  & 0\\
0 & 0 & \gamma 
\end{bmatrix}^n$ converges  to 	$
\begin{bmatrix}
0 & 0 &0\\
0 &0  & 0\\
0 & 0 & 1 
\end{bmatrix}, \textrm{ if  }  \alpha \beta \gamma=1,\, \vert  \alpha \vert < \vert \beta\vert<\vert \gamma\vert  $.

\item  $
\begin{bmatrix}
\alpha & 0 &0\\
0 & e^{2\pi i \theta }\alpha   & 0\\
0 & 0 & \alpha^{-2} e^{-2\pi i \theta } 
\end{bmatrix}^n$ converges  to 	$
\begin{bmatrix}
0 & 0 &0\\
0 &0  & 0\\
0 & 0 & 1 
\end{bmatrix}, \textrm{ if  }  \vert  \alpha \vert < 1$.
	
\end{enumerate}	
\end{proposition}

	A subgroup $G\subset PSL(3,\C)$ is  weakly semi-controllable,  if  it  have a global fixed point $p \in \P^2$; hence  for each line $\mathcal L$ in $\P^2 - \{p\}$ one has a canonical  holomorphic projection map $\pi$ from $\P^2 - \{p\}$ into $\mathcal L \cong \P^1$. This defines a group morphism:
	\[
	\begin{matrix}
	\Pi = \Pi_{p,\ell,G} : G \rightarrow Bihol(\ell)\cong \PSL(2,\C)\\
	\Pi(g)(x) = \pi(g(x))
	\end{matrix}
	\]
	which essentially is independent of all choices, see page 133 in \cite{CNS} for details.\\ 
		\section{Examples} \label{s:examples}
	
	In this section we present examples of curves invariant under Lie groups later we will see that any irreducible curve is one of the curves presented here.
	
	\begin{example}\label{e:line}[Lines]
		The   group of affine   projective transformations. 
	\end{example}
	
	\begin{example}\label{e:ver}[Veronese Curves]
		Recall the Veronese embedding is given by 
		\[
		\begin{array}{l}
		\psi:\Bbb{P}^1_\Bbb{C}\rightarrow \Bbb{P}^2_\Bbb{C}\\
		\psi([z,w])=[z^2,2zw, w^2].
		\end{array}
		\]
		and the unique irreducible representation of $PSL(2,\Bbb{C})$ into $PSL(3,\Bbb{C})$ is given  $\iota: \PSL(2,\Bbb{C})\rightarrow \PSL(3,\Bbb{C})$ where
		\[
		\iota\left(\frac{az+b}{cz+d}\right )=\left [
		\begin{array}{lll}
		a^2&ab&b^2\\
		2ac&ad+bc&2bd\\
		c^2&dc&d^2\\
		\end{array}
		\right ].
		\]
		Then $\iota PSL(2,\Bbb{C})$ leaves invariant  $Ver=\psi(\P^1)$. 
	\end{example}
	\begin{example}\label{e:cubic}[Cubic with a cusp]
		Consider the homogeneous cubic polynomial $p(x,y,z)=xy^2-z^3$ then the projective curve induced by $p$ is a cubic with a cusp in [1:0:0] and a inflection point in  [0:1:0]. Moreover, the cubic and the lines $\overleftrightarrow{e_1,e_2}$, $\overleftrightarrow{e_3,e_2}$ and $\overleftrightarrow{e_1,e_3}$  are  invariant under  the Lie group 
		\[
		\C_p^*=
		\left \{
		\begin{pmatrix}
		a^{-5}&0 &0\\
		0 &a^{4}& 0\\
		0 &0& a\\
		\end{pmatrix}: a\in \C^*
		\right \}
		\]
	\end{example}

	\begin{example} [Pencil of lines] Consider the Lie group given by
\[
\C^2_\infty= 
\left \{
\begin{pmatrix}
1 & a &b \\
0 & 1 & 0\\
0 &0& 1\\
\end{pmatrix}:a,b\in \C
\right \}.
\]
Then $\C^2_\infty$ leaves invariant any   line passing trough $[1:0:0]$.
	\end{example}
	\section{Geometry and dynamic  of  the invariant curves} \label{GDC}
	\begin{lemma} \label{l:fp}
		Let $G\subset PSL(3,\C)$ be a discrete  infinite  group and $S$ an algebraic reducible  complex curve  invariant under $G$. If  $S$ is not a line and  $g\in G$ has infinite order then $S$ contains a fixed point of $g$ and the action of $g$ restricted to $S$    has infinite order. Moreover, if $S$ is non-singular, then $S$ has genus 0.
	\end{lemma}
	
	\begin{proof}
By Lemma \ref{l:inford} there is an element $g\in G$ with infinite order which is either loxodromic or parabolic.  Consider the following  cases:\\
		
		Case 1.- $g$ is loxodromic. By Proposition \ref{p:eq}   there is a   point $p$ an a line $\ell$ such that: $p\notin \ell$,   $p\cup \ell$ is  invariant under $g$ and $(g^n)_{n\in \Bbb{N}}$ converges uniformly on compact sets of $\P^2-\ell$ to $p$. \\
		
		Claim 1.- $p\in S$. Since $S$ is not a line by Bézout Theorem we know that $S\cap \ell$ is a finite set, thus we can find a point $q\in S-\ell$.  Since $S$ is closed and  invariant under $g$, we conclude $p\in S$. \\
		
		Case 2.- $g$ is parabolic.  By Proposition \ref{p:eq}  there is a   point $p$ an a line $\ell$ such that: $p\in \ell$,   $p$, and $ \ell$ are   invariant under $g$ and $(g^n)_{n\in \Bbb{N}}$ converges uniformly on compact sets of $\P^2-\ell$ to $p$. As in the previous  case we can show that $p\in S$.\\

		Finally observe that for any point $q$ in $S-\ell$ the set $\{g^n q\}$ is infinite, which shows that the action of $g$ restricted to $S$ has infinite order.\\
		
	Now let us assume that $S$ is non-singular. If  $\pi_1(S)$ is non-trivial, then there exists a non-trivial class in
		$\pi_1(S)$, say $h$. Let us assume without lost of generality that $G$ contains a loxodromic element $g$,  the proof in the parabolic case will be similar.  Then there  is a   point $p$ an a line $\ell$ such that: $p\notin \ell$,   $p\cup \ell$ is  invariant under $g$ and $(g^n)_{n\in \Bbb{N}}$ converges uniformly on compact sets of $\P^2-\ell$ to $p$.  Recall that $h$ can be assumed as the homotopy  class of a path $\tilde h$ based in $p$, since  $S\cap \ell$ is finite we can assume that  $\tilde h$  is a loop based on $p$ that does not contains points in $\ell$. For $N$ large we have that $g^N(\tilde h)$ is contained in a simply connected neighborhood of $p$ that is $g^N_\#  h$ is trivial, thus $g^{N}_{\#}: \pi_1(S)\rightarrow \pi_1(S) $  is not a group isomorphism, which is a contradiction, since $g^N$ is a homeomorphism. 	
	\end{proof}
	As an immediate consequence of the proof of the previous lemma we have:
	
	\begin{corollary}
If $M\subset \P^2$ an embedded k-manifold with $k\geq 2$ invariant under a discrete group $G\subset PSL(3,\C)$, then $S$ is simply connected
	\end{corollary}
	\begin{lemma} \label{l:genus}
		Let $G\subset PSL(3,\Bbb{C})$ be a discrete group and $S$ a $G$-invariant  complex  irreducible algebraic curve. Then the  dual group $G^*$ leaves invariant the dual algebraic curve $S^*$.  Moreover, if $S$ has singularities then $S^*$ does. 
	\end{lemma}
	\begin{proof}
		Let $p\in S^*$ then  there is a sequence of points $(p_n)\in S^*$ such that each $p_n$ converges to $p$ and for each $n$ we have  $p_n^*$ is a tangent line to $S$ at a non-singular point. Since the action of $G$ is by biholomorphism of $\P^2$ we deduce for each  $\gamma\in\Gamma$   we have $ \gamma( p_n^*)$ is a tangent line to $S$ at a non-singular point. Trivially $\gamma( p_n^*)$ converges to $\gamma( p^*)$.
		
		Let us prove the other part of the Lemma, let us assume that $S^*$ is non-singular, then by the Riemann-Hurwitz formula we have 
		\[
		0=(n-1)(n-2)
		\]
		where $n$ is the degree of $S^*$. Thus $S^*$ is a line  or $S$ is quadratic, since $S^{**}=S$, see page 74 in \cite{Fischer}, we deduce that $S^*$ is  quadratic. By the Plücker class formula, we have  the degree of $S$ is given by $n(n-1)=2$. Since every quadratic in $\P^2$ is projectively equivalent to the Veronese curve and the Veronese curve is non-singular we deduce $S$ is non-singular, which is a contradiction.
	\end{proof}

	\begin{lemma} \label{l:cs}
		Let $G\subset PSL(3,\Bbb{C})$ be a discrete group and $S$ a $G$-invariant  complex  irreducible algebraic curve. If $S$ has singularities then $S$ is  a cubic with one node and one inflection point.
	\end{lemma}
	\begin{proof}  It's well known that $S$ has a finite number of singularities. Since $\Gamma$ acts on $S$ by biholomorphisms of $\P^2$ we conclude that $\Gamma$ takes singularities of $S$ into singularities of $S$, thus  $\Gamma_0=\bigcap_{p\in Sing (S)}Isot(\Gamma_0,p)$, here $Sing(S)$ denotes the singular set of $S$, is a finite index subgroup  of $\Gamma$.\\
		
		Claim 1.-  The genus of $S$ is $0$.  Let $\widetilde{S}$ be the desingularization of $S$, then there is a birrational  equivalence $f:\widetilde{S}\rightarrow  {S}$. Now, let  $\gamma_0 \in \Gamma$ be an element with infinite order, then there is $m\in \Bbb{N}$ such that $\gamma_1=\gamma_0^n\in \Gamma_1$. since $f$ is a birrational equivalence we can construct $\widetilde {\gamma}_{1}:\widetilde{S}-f^{-1}Sing (S) \rightarrow \widetilde{S}-f^{-1}Sing (S)$ a biholomorphism such that the following diagram commutes:
		
		\begin{equation}\label{e:diag}
			\xymatrix{\widetilde{S}-f^{-1}Sing (S)\ar[r]^{\widetilde {\gamma}_{1}}\ar[d]^{f}& \widetilde{S}-f^{-1}Sing (S)\ar[d]^{f}\\
				S-Sing(S)\ar[r]^{\gamma_1} & S-Sing(S)}
		\end{equation}
		
		Since each arrow in the previous diagram is a biholomorphism we deduce that $\widetilde {\gamma}_{1}$ admits a holomorphic extension to $\widetilde{S}$. Observe that  by Lemma \ref{l:fp}, $\gamma_1$ has a fixed point in $S$ and its action on $S$ has infinite order, then by diagram \ref{e:diag}  we deduce that $\widetilde {\gamma}_{1}$ has infinite order and at least one fixed point. Recall that a Riemann surface whose group of biholomorphisms is infinite should have genus 1 or 0, see Theorem in 3.9 in  \cite{miranda}, and in the case of Riemann surfaces of genus 1 the subgroup of biholomorphism  sharing a fixed point should be finite, see Proposition 1.12 in \cite{miranda}, which concludes the proof of the claim.\\

		Claim 2.-  The curve  $S$ has at most two singularities. Moreover, the singular set is either a node or at most  simple cusp. Let $f$, $\widetilde{S}$, $\gamma_1$, $\widetilde \gamma_{1}$ as in claim 1.  Thus $\widetilde \gamma_{1}$ fixes each point in $f^{-1}(Sing(S))$. Since $\widetilde S$ has genus 0,  we deduce  $\widetilde \gamma_{1}$  can be   conjugated to a Möbius transformation. Because $\widetilde \gamma_{1}$  has infinite order, we conclude $f^{-1}(Sing(S_j))$ contains at most two points.  If $p$ is a singular point in $S$ we have $f^{-1}{p}$ is either one point or two points. If it is one point, $p$ should be a cusp and in the remaining case   $p$  is a simple node. Now  it is clear that either $S$ has one simple node or at most two cusp.\\


		Claim 3.- $S$ has degree 3 and the singular set is either a simple node or a single cusp. Given that  $Sing(S)$ contains  only cusp and simple nodes and the  genus is  0, by applying Clebsch's genus formula, see page 179 in \cite{Fischer},  to $S$ and we get:
		\[
		0=genus(S)=(n-1)(n-2)-2(d+s)
		\]
		where $n$ is the degree,  $d$ is the number of nodes and  $s$ is the number of cusp  in $S$. In our case the previous equation implies  following possibilities:
		\[
		\left\{
		\begin{array}{ll}
		d=1     & s=0 \\
		d=0     & s=1\\
		d=0 & s=2
		\end{array}
		\right.
		\]
		Substituting this  values in the  Clebsch's genus formula we get $n=3$ and  also we conclude  the case  $d=0, s=2$ is not possible.\\
		
		On the hand, by Lemma \ref{l:genus} the curve  $S^*$ is singular and has degree three, thus by  Plücker class formula, see page 89 in \cite{Fischer}, we obtain:
		\[
		3=deg(S^*)=deg(S)(deg(S)-1)-2d-3s=6-2d-3s
		\]
		which is only possible when $d=0$ and $s=1$, that is  the singular set consist of a single cusp. To conclude the proof we need to use Plücker inflection point formula, recall  this formula is given by,  see page 89 in \cite{Fischer}:  $s^*=3deg(S)(deg(S)-2)-6d-8s$, where $s^*$ is the number on inflection points in $S$.
	\end{proof}
	\begin{lemma} \label{l:cc}
		Let $G\subset PSL(3,\Bbb{C})$ be a discrete group and $S$ be an irreducible singular curve invariant under $G$. Then  there is a a projective transformation $\gamma\in PSL(3,\Bbb{C})$ such that $gS$ is the curve induce by the polynomial $p(x,y,z)=xy^2-z^3$ and $\gamma G\gamma^{-1}\subset \Bbb{C}^*_p$, see example \ref{e:cubic}.  
	\end{lemma}
	\begin{proof}
		By Lemma \ref{l:cs} we have $S$ is a cubic with a cusp, then there is a projective transformation $\gamma$ such that $\gamma S$ is the Curve induced by the polynomial $p(x,y,z)=xy^2-z^3$, see cite \cite{sha}. Since $G$ acts by biholomorphism of $\Bbb{P}^2_\Bbb{C}$ the group  $gGg^{-1}$ leaves invariant the singular point and the inflection point of $gS$. On the other hand, 
		a  straightforward computation shows  $\overleftrightarrow{e_2,e_3}$ is the tangent line to $gS$ at $[e_2]$ and $\overleftrightarrow{e_1,e_3}$ is the unique  tangent line to $gS$ at $[e_1]$, once again since the action  of $G$ on $\Bbb{P}^2_\Bbb{C}$  is biholomorphisms we conclude that $\overleftrightarrow{e_2,e_3}$  and $\overleftrightarrow{e_1,e_3}$  are G-invariant, thus $[e_3]$  is fixed by $gGg^{-1}$. Therefore each element in $gGg^{-1}$ has a diagonal lift. Let $g\in gGg^{-1}$ and $(g_{ij})\in SL(3,\Bbb{C})$ be a lift of $g$, considering  $[1:1:1]\in gS$  we conclude:
		\[
		g_{11}g_{22}^2-g_{33}^3=0
		\]
		Using the fact $g_{11}g_{22}g_{33}=1$, we deduce $g_{22}=g_{33}^4$ and $g_{11}=g_{33}^{-5}$, which concludes the proof.
	\end{proof}
	
\section{Proof of the main Theorems} \label{s:main}
Since $S$ is an algebraic surface, we know $S$ is a finite union of irreducible curves, say $S=\bigcup_{j=1}^n S_j$, then $G_0=\bigcap Isot(G, s_j)$ is a finite index subgroup of $G$, leaving invariant each $S_j$. If $S_j$ has singularities  then by Lemmas \ref{l:cs} and \ref{l:cc}  we have that $S_j$ is projectively equivalent to the curve induced by  $xy^2-z^3$ and the group $G_0$ is virtually cyclic. If $S_j$ is non- singular by the Riemann Hurwitz theorem and Lemma \ref{l:fp} we deduce that $S_j$ is either a line or a copy of the Veronese curve, this shows Theorem \ref{t:main1}.

In order to prove Theorem \ref{t:main2}. Observe, the previous argument also shows that if $G$ is non-virtually cyclic, then each connected component of $S$ is either a line or a copy of the Veronese curve. Let us assume that  $S$ contains at least two irreducible component, say $S_0$ and $S_1$, also let us assume that  $S_0$ is projectively equivalent   to the Veronese curve, then $S_0\cap S_1$ is finite, non-empty and $G_0$ invariant. On the other hand, since $G_0$ leaves invariant $S_0$ and $S_0$ is biholomorphicaly equivalent to the sphere Lemma \ref{l:fp} ensures that $G_0$ is virtually commutative, which is a contradiction. Thus if $S$ contains a Veronese curve the curve is exactly the Veronese curve, if each reducible component is a line then by proposition 5.15 in  \cite{BCN}   the largest number of lines in general position in $S$ is 3.

Now the proofs of Theorem \ref{t:main3} and Corollary \ref{c:main} are trivial, so we omit it  here.$\square$\\

It would be interesting to understand  how this result  changes  in the higher dimensional setting as well as in the case   when  we consider  smooth manifold or real algebraic manifold as the  invariant sets.
	
	
	\section*{Acknowledgments}
	
	The authors would like to thank to the UCIM UNAM and
	their people for its hospitality and kindness during the writing
	of this paper. We also grateful to J. F. Estrada and J. J. Zacarías   for fruitful conversations.

	\bibliographystyle{amsplain}

\begin{thebibliography}{10}
		\bibitem{BCN}
		W. Barrera,  A. Cano, J. P. Navarrete, 
{\ On the number of lines in the limit set for discrete subgroups of}  $PSL(3,\C)$, Pacific Journal of Mathematics, 
Vol. 281 (2016), No. 1, 17–49
	
		\bibitem{BCNS}
		W. Barrera,  A. Cano, J. P. Navarrete, and J. Seade, {\it Reducible Complex Kleinian groups in} $PSL(3,\Bbb{C})$,  in preparation.
		
		\bibitem{BCNS1}
		 W.  Barrera; A. Cano;   J.  P.  Navarrete y J. Seade, {\it Purely parabolic discrete  groups pf $PSL(3,\Bbb{C})$}, preprint arxiv 1802.08360, 2020,
		
		\bibitem{BCS}
		J. Y.  Briend,  S.  Cantat,  M.  Shishikura, {\it  Linearity of the exceptional set for maps of}  $P_ k (C)$, M. Math. Ann. (2004) 330, Issue  1,   39-43.
		
		
		\bibitem{CL}
		A. Cano, L. Loeza, 
		{\it Two dimensional Veronese groups with an invariant ball}, International Journal of Mathematics Vol. 28, No. 10 (2017) 1750070 (17 pages), DOI: 10.1142/S0129167X17500707.
		
		\bibitem{CLU} 
		A. Cano, L. Loeza, A. Ucan-Puc, {\em Projective cyclic groups in higher dimensions}. Linear Algebra and its Applications 
		531 (2017) 169-209. 
		
	\bibitem{CS}
		A. Cano;  J. Seade, 
		{\it On discrete groups of automorphisms of $\Bbb{P}^2_{\Bbb{C}}$}, 
		Geometriae Dedicata, February 2014, Volume 168, Issue 1, pp. 9-60.
		
		\bibitem{CNS}
		A. Cano, J. P. Navarrete, and J. Seade, {\it Complex Kleinian groups}, Progress in Mathematics,
		no. 303, Birkhauser/Springer Basel AG, Basel, 2013.
		
\bibitem{CLi}
D. Cerveau, A. Lins,  {\ Hypersurfaces exceptionnelles des endomorphismes de} $\Bbb{CP}^n$, Boletim Da Sociedade Brasileira de Matematica, 31(2), 155–161, 2000.
		
		\bibitem{Fischer}
		G. Fischer, {\it Plane Algebraic Curves}, Student Mathematical Library
		Vol 15, AMS, 2001.
		
		\bibitem{FS}
		J.F. Fornaes, N. Sibony, {\it  Complex dynamic in higher dimension I}, Astérisque,222: 201–231.
		
		\bibitem{Gr}
		L. Greenberg, {\it Discrete subgroups of the Lorentz group}, Math. Scand. 10 (1962), 85–107
		
		\bibitem{Kap}
			M. Kapovich,
		Hyperbolic Manifolds and Discrete Groups,
Modern Birkhäuser Classics book series, 2010.
		
\bibitem{Ku} R. S. Kulkarni, {\it 
	Groups with Domains of Discontinuity}. Mathematische Annalen No. 237 (1978).
pp. 253-272.		
		\bibitem{miranda}
		R. Miranda, {\it  Algebraic Curves and Riemann Surfaces},  Graduate Studies in Mathematics Vol 5,  AMS, 1995.



\bibitem{Gus}	
J. P. Navarrete, {\it The trace function and Complex Kleinian groups},	International Journal of Mathematics, Vol. 19, No. 07, pp. 865-890 (2008).
		\bibitem{SV}
J. Seade, A. Verjovsky, {\it Higher dimensional complex Kleinian groups}, Mathematische Annalen. 322. 279-300.

		\bibitem{sha}
		I. R. Shafaverich,  Basic Algebraic Geometry 1, Springer-Verlag Berlin Heidelberg
		Copyright Holder, 1994.
	\end{thebibliography}

\end{document}